\documentclass[12pt]{amsart}
\usepackage{amssymb,amsmath,mathrsfs}
\usepackage[toc,page]{appendix}
\usepackage[ps2pdf,colorlinks=true,urlcolor=blue,
citecolor=red,linkcolor=blue,linktocpage,pdfpagelabels,bookmarksnumbered,bookmarksopen]{hyperref}
\usepackage{epsfig,graphicx,color,mathrsfs}
\usepackage[english]{babel}
\usepackage[left=2.7cm,right=2.7cm,top=2.6cm,bottom=2.6cm]{geometry}

\newcommand{\N}{{\mathbb N}}

\newcommand{\R}{{\mathbb R}}

\renewcommand{\a }{\alpha}

\renewcommand{\d }{\delta }

\renewcommand{\O}{\Omega}

\numberwithin{equation}{section}
\newtheorem{theorem}{Theorem}[section]
\newtheorem{proposition}[theorem]{Proposition}
\newtheorem{lemma}[theorem]{Lemma}

\newtheorem{definition}[theorem]{Definition}
\theoremstyle{definition}

\newcommand{\brm}{\begin{remark}\rm}
\newcommand{\erm}{\end{remark}}
\newcommand{\brms}{\begin{remark}\rm}
\newcommand{\erms}{\end{remark}}
\newcommand{\bte}{\begin{theorem}}
\newcommand{\ete}{\end{theorem}}
\newcommand{\bpr}{\begin{proposition}}
\newcommand{\epr}{\end{proposition}}
\newcommand{\ble}{\begin{lemma}}
\newcommand{\ele}{\end{lemma}}
\newcommand{\beq}{\begin{equation}}
\newcommand{\eeq}{\end{equation}}
\newcommand{\bdm}{\begin{displaymath}}
\newcommand{\edm}{\end{displaymath}}
\numberwithin{equation}{section}

\newcommand{\bos}{\begin{remark}\rm}
\newcommand{\eos}{\end{remark}}

\newcommand{\ben}{\begin{enumerate}}
\newcommand{\een}{\end{enumerate}}

\renewcommand{\a }{\alpha }

\renewcommand{\d}{\delta }
\newcommand{\D }{\Delta }

\newcommand{\n }{\nabla }

\renewcommand{\O }{\Omega }
\newcommand{\vp }{\varphi }

\newcommand{\ov}{\overline}
\newcommand{\pa}{\partial}

\newcommand{\wt }{\tilde}

\newcommand{\be}{\begin{equation}}
\newcommand{\ee}{\end{equation}}

\newcommand{\DD}{{\mathbb H}}

\title[Symmetry for nonvariational quasi-linear elliptic systems]{Symmetry
results for nonvariational \\ quasi-linear elliptic systems}

\author[L.\ Montoro]{Luigi Montoro$^*$}
\address{Dipartimento di Matematica
\newline\indent
Universit\`a della Calabria
\newline\indent
Ponte Pietro Bucci 31B, I-87036 Arcavacata di Rende, Cosenza, Italy}
\email{montoro@mat.unical.it}

\author[B.\ Sciunzi]{Berardino Sciunzi$^*$}
\address{Dipartimento di Matematica
\newline\indent
Universit\`a della Calabria
\newline\indent
Ponte Pietro Bucci 31B, I-87036 Arcavacata di Rende, Cosenza, Italy}
\email{sciunzi@mat.unical.it}

\thanks{$^*$Dipartimento di Matematica,
Universit\`a della Calabria,
Ponte Pietro Bucci 31B, I-87036 Arcavacata di Rende, Cosenza, Italy,
E-mail: {\em montoro@mat.unical.it}, {\em sciunzi@mat.unical.it}}

\author[M.\ Squassina]{Marco Squassina$^\dagger$}
\address{Dipartimento di Informatica
\newline\indent
Universit\`a degli Studi di Verona
\newline\indent
C\'a Vignal 2, Strada Le Grazie 15, I-37134 Verona, Italy}
\email{marco.squassina@univr.it}

\thanks{$^\dagger$Departimento di Informatica,
Universit\`a di Verona,
C\`a Vignal 2, Strada Le Grazie 15, I-37134 Verona, Italy.
E-mail: {\em marco.squassina@univr.it}}

\thanks{The authors were partially supported by the
Italian PRIN Research Project 2007: {\em Metodi Variazionali e Topologici
nello Studio di Fenomeni non Lineari}}

\begin{document}



\begin{abstract}
By virtue of a weak comparison principle in small domains we prove axial symmetry in convex
and symmetric smooth bounded domains as well as radial symmetry in balls for regular solutions
of a class of quasi-linear elliptic systems in non-variational form. Moreover, in the two
dimensional case, we study the system when set in a half-space.
\end{abstract}
\maketitle

\section{Introduction and main results}
The aim of this paper is to get some symmetry and monotonicity results for the solutions  $(u,v)\in C^{1,\a}(\ov{\O})\times C^{1,\a}(\ov{\O})$
to the following quasi-linear elliptic system
\begin{equation}
	\label{NVsystem}
\begin{cases}
-\D_p u= f(u,v) & \text{in $\O$},  \\
-\D_m v= g(u,v) & \text{in $\O$},  \\
u>0,\,\, v>0  & \text{in $\O$},    \\
u=0,\,\, v=0 & \text{in $\pa\O$},
\end{cases}
\end{equation}
where $\Omega$ is a smooth bounded domain of $\R^N$, $N\geq 2$ and
$\D_p={\rm div}(|Du|^{p-2}Du)$ is the $p$-Laplacian operator, $|\cdot|$ denoting
the standard Euclidean norm in $\R^N$. Furthermore, in the
two-dimensional case, we shall also consider the system defined in the
half-space.
Problem \ref{NVsystem} is the stationary system corresponding to the parabolic system
\begin{equation*}
\begin{cases}
u_t-\D_p u= f(u,v) & \text{in $\O\times (0,\infty)$},  \\
v_t-\D_m v= g(u,v) & \text{in $\O\times (0,\infty)$},  \\
\end{cases}
\end{equation*}
where  the  adoption of the $p$-Laplacian operator inside the diffusion term arises
in various applications where the standard linear heat operator $u_t-\Delta$
is replaced by a nonlinear diffusion with gradient dependent diffusivity. The equations in the above system usually arise
in the theory of non-Newtonian filtration fluids, in turbulent flows in porous media
and in glaciology (cf.~\cite{AE}).

System~\eqref{NVsystem}
does not necessarily admits a variational structure and it has been previously studied
in the literature both from the point of view of existence and symmetry of smooth solutions.
For the existence of a positive radially symmetric $C^2$ solution in the particular case
where $f(u,v)=u^\alpha v^\beta$ and $g(u,v)=u^\gamma v^\delta$ for suitable values
of $\alpha,\beta,\gamma,\delta\geq 0$, we refer the reader to~\cite{CFMT} and to the reference therein. Concerning
the symmetry properties (and a priori estimates) of any smooth solution of~\eqref{NVsystem} in
the special case $f(u,v)=f(v)$ and $g(u,v)=g(u)$ are positive and nondecreasing functions,
we refer to~\cite{DS3} (see also~\cite{ACM}).

In our main results we shall always assume on $f,g$ that
\begin{equation}
	\label{libandpos}
	f,g\in {\rm Lip}_{\rm loc}(\R^2_+)
	\quad\text{and}\quad
	f(s,t)>0,\quad g(s,t)>0,\quad\text{for all $s,t>0$},
\end{equation}
and that they satisfy the monotonicity (also known as {\em cooperativity}) conditions
\begin{equation}
	\label{coop}
\frac{\partial f}{\partial t}(s,t)\geq 0
\quad\text{and}\quad
\frac{\partial g}{\partial s}(s,t)\geq 0,\quad\text{for all $s,t>0$}.
\end{equation}
The sign assumptions~\eqref{libandpos} and~\eqref{coop} are natural in the study of this class of problems. Furthermore,
it is shown in~\cite{Tr} that conditions~\eqref{coop} are, actually, necessary in order to obtain
symmetry results for the solutions to~\eqref{NVsystem}. For useful regularity features of
the solutions to~\eqref{NVsystem}, we refer the reader to~\cite[Section 2]{DS3} where the regularity
of the quasi-linear equation $-\D_p u=h(x)$ is investigated under the assumption that $h\in C^{0,\alpha}
\cap W^{1,\sigma}_{{\rm loc}}(\Omega)$, where $\sigma\geq\max\{N/2,2\}$. In turn, the regularity properties of~\eqref{NVsystem}
can be obtained by applying the results of~\cite{DS3} to the choices $h(x)=f(u(x),v(x))$ and $h(x)=g(u(x),v(x))$
where $f,g$ are locally Lipschitz.
\vskip2pt
\noindent
Under the same cooperativity condition \eqref{coop}, for the non-degenerate case $p=2=m$, we refer e.g. to \cite{BS,De,Tr} and references included.
\vskip2pt
\noindent
In the following we present our symmetry results, which complete those of~\cite{DS3},
first in the case where system~\eqref{NVsystem} is set is a
smooth bounded symmetric domain and, then, when it is set in a half-space of $\R^2$.
\vskip8pt
\noindent Our results are based on the use of a refined version of the Moving plane technique \cite{S} (see also \cite{GNN}). We will in particular use
the moving plane procedure as improved in \cite{BN}.  In the case of the half-space of $\R^2$, we exploit a geometric idea as in \cite{DS4}, which
is more related to the techniques developed in \cite{BCN}.

\subsection{System in a smooth bounded domain}
In a bounded domain $\Omega$,
we consider solutions $u,v \in C^1(\ov{\O})\times C^1(\ov{\O})$
to the non-variational quasi-linear system
\begin{equation}\label{eq:sist}
\begin{cases}
-\D_p u= f(u,v) & \text{in $\O$},  \\
-\D_m v= g(u,v) & \text{in $\O$},  \\
u>0,\,\, v>0  & \text{in $\O$},    \\
u=0,\,\, v=0 & \text{in $\pa\O$},
\end{cases}
\end{equation}
Furthermore, we assume that~\eqref{libandpos} and that the cooperativity
condition~\eqref{coop} is satisfied.
Let us set
\begin{equation*}
Z_u\equiv \{x\in \Omega: \n u(x)=0\},\qquad Z_v\equiv \{x\in \Omega: \n v(x)=0\}.
\end{equation*}

The first main result of the paper is the following

\begin{theorem}
	\label{main1}
Assume that~\eqref{libandpos} and~\eqref{coop} hold.
If $\Omega$ is convex with respect to the  $x_1$-direction, and  symmetric with
respect to the hyperplane $T_{0}=\{x_1  =0\}$,
then $u$ and $v$ are symmetric and nondecreasing in
the $x_1$-direction in  $\Omega _0=\{x_1  <0\}$, with
$$
\frac {\partial  u}{\partial x_1}(x) >0\quad\text{in $ \Omega  _0  \setminus Z_u $},
\qquad
\frac {\partial  v}{\partial x_1}(x) >0\quad\text{in $ \Omega  _0  \setminus Z_v $}.
$$
In particular, if $\Omega  $ is a ball, then $u$ and $v$ are radially symmetric with
$\frac {\partial  u}{\partial  r}(r)<0$ and $\frac {\partial  v}{\partial  r}(r)<0$.
\end{theorem}

Notice that this result holds true under the same assumptions that were considered
in~\cite{DS3} where the particular case $f(u,v)=f(v)$ and $g(u,v)=g(u)$ is considered. More precisely,
no monotonicity is requested on the function $f$ (resp.\ $g$)
with respect to $u$ (resp.\ $v$).
\vskip3pt

The second result is an improvement under some restrictions on the values of $p,m$, of the previous Theorem~\ref{main1}.

\begin{theorem}
		\label{main2}
	Assume that~\eqref{libandpos} and~\eqref{coop} hold and  $\frac{2N+2}{N+2}<p ,m<\infty$.
		  If $\Omega$ is convex with respect to the $x_1$-direction and  symmetric with
  respect to the hyperplane $T_{0}=\{x_1=0\}$,   then $u$ and $v$ are
  symmetric and nondecreasing in the $x_1$-direction in  $\Omega_0=\{x_1<0\}$ with
    $$
\frac {\partial  u}{\partial  x_1}(x) >0\quad\text{in $\Omega_0$},\qquad
\frac {\partial  v}{\partial  x_1}(x) >0\quad\text{in $\Omega_0$}.
$$
In particular $Z_u\subset T_0$ and $Z_v\subset T_0$ . Therefore if for $N$ orthogonal directions
$e_i $ the domain $\Omega$ is symmetric with respect to any hyperplane ${\displaystyle T_{0}^{e_i }
=\{x \cdot e_i=0\}}$, then
\begin{equation}
	\label{zxcPRIbis}
Z_u=Z_v=\{0\},
\end{equation}
assuming that $0$ is the center of symmetry.
\end{theorem}

\medskip
\subsection{System on a half-space of $\R^2$}
Let $\DD=\{(x,y)\in \R^2: y > 0\}$ and consider the system
\begin{equation}
	\label{eq:sist:semi}
\begin{cases}
-\D_p u= f(u,v) & \text{in $\DD$}, \\
-\D_m v= g(u,v) & \text{in $\DD$}, \\
u>0,\,\, v>0  & \text{in $\DD$},   \\
u=0,\,\, v=0 & \text{on $\pa\DD$}.
\end{cases}
\end{equation}

Then we have the following monotonicity result

\begin{theorem}
	\label{main3}
Let $(u,v)$ be a  nontrivial weak
$C^{1,\alpha}_{{\rm loc}}(\DD)$ solution of~\eqref{eq:sist:semi}.
Assume that~\eqref{libandpos} and~\eqref{coop} hold and let  $\frac{3}{2}<p,m<\infty$.
Then
$$
\frac{\partial u}{\partial y}(x,y) > 0\quad\text{and}\quad
\frac{\partial v}{\partial y}(x,y) > 0
\qquad \text{for all $(x,y) \in \overline{\DD}$}.
$$
\end{theorem}

We prove Theorem \ref{main3} by exploiting a weak comparison principle in small
domains (see Proposition \ref{pro:confr}), and some techniques developed
in \cite{DS4}, where the monotonicity of the solutions was used to prove some
Liouville type theorems for Lane-Emden-Fowler type equations.

\bigskip
\vskip8pt
\begin{center}\textbf{Notations.}\end{center}
\begin{enumerate}
\item For $n\geq 1$, we denote by $|\cdot|$ the euclidean norm in $\R^n$.
\item $\R^+$ (resp.\ $\R^-$) is the set of positive (resp.\ negative) real values.
\item For $p>1$ we denote by $L^p(\R^n)$ the space of measurable functions $u$ such that
$\int_{\O}|u|^pdx<\infty$. The norm $(\int_{\O}|u|^pdx)^{1/p}$ in $L^p(\O)$ is denoted by $\|\cdot\|_{L^p(\O)}$.
\item For $s\in\N$, we denote by $H^s(\O)$ the Sobolev space of functions $u$ in $L^2(\O)$
having generalized partial derivatives $\partial_i^k u$ in $L^2(\O)$ for all $i=1,\dots, n$ and
any $0\leq k \leq s$.
\item The norm $(\int_{\O}|u|^pdx+ \int_{\O}|\nabla u|^pdx)^{1/2}$ in $W^{1,p}_0(\O)$ is denoted by $\|\cdot\|_{W^{1,p}_0(\O)}$.
\item We denote by $C_0^{\infty}(\O)$ the set of smooth compactly supported functions in $\O$.
\item We denote by $B(x_0,R)$ a ball of center $x_0$ and radius $R$.
\item We denote by $\mathcal{L}(E)$ the Lebesgue measure of the set $E\subset\R^n$.
\end{enumerate}
\medskip

\section{Proof of the results}

In the next section we shall prove the main results of the paper.

\subsection{Proof of Theorem~\ref{main1}}

First, we have the following weak comparison principle in
small  sub-domains $\Omega_0$ of $\Omega$.

\begin{proposition}
	\label{pro:confr}
Assume that $u,v \in C^1(\ov{\O})$ and $\tilde u, \tilde v \in C^1(\ov{\O})$
are solutions to~\eqref{eq:sist}. Let $\O_0$ be a bounded smooth domain of $\R^N$ such that $\O_0\subset \O$.
Then there exists a positive number $\d$, depending upon $f,g$, $\|u\|_\infty ,\|v\|_\infty,\|\tilde u\|_\infty,\|\tilde v\|_\infty$, such that if
$$
\mathcal{L}(\O_0) \leq \d,
$$
and
$$
u\leq \tilde u\quad\text{ on $\pa\O_0$}, \qquad
v\leq \tilde v\quad\text{ on $\pa\O_0$},
$$
then
$$
u\leq\tilde u\quad\text{ on $\O_0$}, \qquad
v\leq\tilde v\quad\text{ on $\O_0$}.
$$
\end{proposition}
\begin{proof}
We consider four different cases:
\begin{enumerate}
  \item $p > 2$ and $m>2$;
  \item $p \leq 2$ and $m>2$;
  \item $p>2$ and $m\leq 2$;
  \item $p < 2$ and $m < 2$.
\end{enumerate}
We will show that the result follows in cases (1) and (2), the others cases being similar.
We will denote by $C$ a generic positive constant, which may change from line to line throughout
the proof.
\vskip2pt
\noindent
{\bf Case 1.} ($p > 2$ and $m>2$). Let us set
$$
U=(u- \tilde{u})^+\quad\text{and}\quad V=(v-\tilde{v})^+.
$$
We will prove the result by showing that, actually, it holds $U\equiv V\equiv 0$.
Since both $u \leq\tilde{u}$ on $\pa\O_0$ and $v \leq \tilde{v}$ on $\pa\O_0$ then
the functions $U, V$ belong to $W^{1,p}_0(\O_0)$. Therefore, let us consider the
variational formulations of the equations of~\eqref{eq:sist}.
\begin{align}
	\label{eq:wu}
& \int_\O|\n u|^{p-2}(\n u,\n \vp) dx= \int_\O f(u,v) \vp dx, \quad
\forall \vp \in C^\infty_c(\O), \\
	\label{eq:wtu}
& \int_\O|\n \wt{u}|^{p-2}(\n \wt{u},\n \vp) dx= \int_\O
f(\wt{u},\wt{v}) \vp dx, \quad \forall \vp \in C^\infty_c(\O), \\
&	
\label{eq:wv}
\int_\O|\n v|^{m-2}(\n v,\n \vp) dx= \int_\O g(u,v) \vp dx, \quad
\forall \vp \in C^\infty_c(\O), \\
&	
\label{eq:wtv}
\int_\O|\n \wt{v}|^{m-2}(\n \wt{v},\n \vp) dx= \int_\O
g(\wt{u},\wt{v}) \vp dx, \quad \forall \vp \in C^\infty_c(\O).
\end{align}
By a density argument, we can put respectively $\vp=U$ in
equations~\eqref{eq:wu} and~\eqref{eq:wtu} and $\vp=V$ in
equations~\eqref{eq:wv} and~\eqref{eq:wtv}. Subtracting, we get
\begin{align}
	\label{eq:sub1}
& \int_{\O_0}\big( |\n u|^{p-2}\n u -|\n\wt{u}|^{p-2}\n \wt{u},
\n (u-\wt{u})^+
\big)dx=\int_{\O_0}[f(u,v)-f(\wt{u},\wt{v})](u-\wt{u})^+dx, \\
	\label{eq:sub2}
& \int_{\O_0}\big( |\n v|^{m-2}\n v -|\n\wt{v}|^{m-2}\n \wt{v}, \n
(v-\wt{v})^+
\big)dx=\int_{\O_0}[g(u,v)-g(\wt{u},\wt{v})](v-\wt{v})^+dx.
\end{align}
Now we use the following standard estimate
$$
(|\eta|^{q-2}\eta-|\eta'|^{q-2}\eta',\eta-\eta')\geq C (|\eta|+|\eta'|)^{q-2}|\eta-\eta'|^2,
$$
for all $\eta,\eta' \in \R^N$ with $|\eta|+|\eta'|>0$ and $q>1$,
from equations~\eqref{eq:sub1} and~\eqref{eq:sub2} one has that
\begin{align}
	\label{eq:sub11}
& \int_{\O_0}( |\n u|+|\n\wt{u}|)^{p-2}|\n (u-\wt{u})^+|^2
dx\leq C\int_{\O_0}[f(u,v)-f(\wt{u},\wt{v})](u-\wt{u})^+dx, \\
	\label{eq:sub22}
& \int_{\O_0}( |\n v|+|\n\wt{v}|)^{m-2}|\n (v-\wt{v})^+|^2
dx\leq C\int_{\O_0}[g(u,v)-g(\wt{u},\wt{v})](v-\wt{v})^+dx.
\end{align}
Since $f$ is locally lipschitz continuous, from equation~\eqref{eq:sub11} it follows
\begin{align}
	\label{eq:sub111}\nonumber
\int_{\O_0} |\n u|^{p-2}|\n (u-\wt{u})^+|^2 dx
&\leq
C\int_{\O_0}\Big[\frac{f(u,v)-f(\wt{u},v)}{u-\wt{u}}\Big]((u-\wt{u})^+)^2dx\\\nonumber
&+C\int_{\O_0}\Big[\frac{f(\wt{u},v)-f(\wt{u},\wt{v})}{(v-\wt{v})^+}\Big](u-\wt{u})^+(v-\wt{v})^+dx\\
\nonumber &\leq C\Big
(\int_{\O_0}((u-\wt{u})^+)^2dx+\int_{\wt{\O}}(u-\wt{u})^+(v-\wt{v})^+dx\Big)
\\ &\leq C
\Big(\int_{\O_0}((u-\wt{u})^+)^2dx+\int_{\wt{\O}}((v-\wt{v})^+)^2dx\Big),
\end{align}
where, of course, in the last inequality we have used Young's inequality.
Arguing in the same fashion, since $g$ is locally lipschitz continuous as well,
from equation~\eqref{eq:sub22} one deduces
\begin{equation}\label{eq:sub222}
\int_{\O_0} |\n v|^{m-2}|\n (v-\wt{v})^+|^2 dx\leq C
\Big(\int_{\O_0}((u-\wt{u})^+)^2dx+\int_{\wt{\O}}((v-\wt{v})^+)^2dx\Big).
\end{equation}
We know that a weighted Poincar\'{e} inequality holds true (cf.~\cite{DS1}), that yields
\begin{align}
	\label{eq:pu}
& \int_{\O_0} ((u-\wt{u})^+)^2dx\leq C_1(\O_0)\int_{\O_0} |\n
u|^{p-2}|\n (u-\wt{u})^+|^2dx, \\
	\label{eq:pv}
& \int_{\O_0} ((v-\wt{v})^+)^2dx\leq C_2(\O_0)\int_{\O_0} |\n
v|^{m-2}|\n (v-\wt{v})^+|^2dx,
\end{align}
where $C_1(\O_0)\to 0$, when $\mathcal{L}(\O_0)\to 0$, as well as $C_2(\O_0)\to 0$, for
$\mathcal{L}(\O_0)\to 0$. In turn, by combining inequalities~\eqref{eq:sub111}
and~\eqref{eq:sub222}, and setting
$$
C_{\Omega_0}=C\max\{C_1(\O_0),C_2(\O_0)\},
$$
we conclude that
\begin{align*}
& \int_{\wt{\O}} |\n u|^{p-2}|\n (u-\wt{u})^+|^2 dx\leq
C_{\Omega_0}\Big(\int_{\wt{\O}}|\n u|^{p-2}|\n(u-\wt{u})^+|^2dx+\int_{\wt{\O}}|\n
v|^{m-2}|\n(v-\wt{v})^+|^2dx\Big), \\
& \int_{\wt{\O}} |\n v|^{m-2}|\n (v-\wt{v})^+|^2 dx\leq
C_{\Omega_0}\Big (\int_{\wt{\O}}|\n
u|^{p-2}|\n(u-\wt{u})^+|^2dx+\int_{\wt{\O}}|\n
v|^{m-2}|\n(v-\wt{v})^+|^2dx\Big).
\end{align*}
By adding these equations, and setting
$$
I(\O_0)=\int_{\O_0} |\n u|^{p-2}|\n (u-\wt{u})^+|^2 dx+\int_{\O_0} |\n v|^{m-2}|\n (v-\wt{v})^+|^2 dx,
$$
we obtain
\begin{equation}
	\label{eq:sub12}
I(\O_0)\leq C_{\Omega_0} I(\O_0).
\end{equation}
Now, we choose the value of $\d>0$ so small that the condition $\mathcal L(\O_0)\leq\d$ implies
$C_{\Omega_0}<1$. Therefore, from equation~\eqref{eq:sub12}, we get the
desired contradiction. In turn, we get
$$
(u-\wt{u})^+ \equiv 0 \quad\text{and} \quad (v-\wt{v})^+ \equiv 0,
$$
concluding the proof in this case.
\vskip4pt
\noindent
{\bf Case 2.} ($p \leq 2$ and $m>2$). Since $p\leq 2$ and $u \in C^{1, \a}(\ov{\O})$,
then equation~\eqref{eq:sub11} gives
\begin{equation}
\label{eq:sub11c2}
\int_{\O_0}|\n (u-\wt{u})^+|^2 dx\leq C\int_{\O_0}[f(u,v)-f(\wt{u},\wt{v})](u-\wt{u})dx.
\end{equation}
Then, since $f$ is locally lipschitz continuous, via the standard Poincar\'{e}
inequality and the weighted Poincar\'e inequality~\eqref{eq:pv}, from
inequality~\eqref{eq:sub11c2} one has
\begin{equation*}
\int_{\O_0}|\n (u-\wt{u})^+|^2 dx\leq CC_1(\O_0)\Big
(\int_{\O_0}|\n(u-\wt{u})^+|^2+\int_{\O_0}|\n v|^{m-2}|\n
(v-\wt{v})^+|^2dx\Big).
\end{equation*}
In the very same way, one gets
\begin{equation*}
\int_{\O_0} |\n v|^{m-2}|\n (v-\wt{v})^+|^2 dx \leq
CC_2(\O_0)
\Big(\int_{\O_0}|\n(u-\wt{u})^+|^2dx+\int_{\O_0}|\n
v|^{m-2}|\n(v-\wt{v})^+|^2dx\Big).
\end{equation*}
Adding these equations, setting
$$
J(\O_0)=\int_{\O_0}|\n (u-\wt{u})^+|^2+\int_{\O_0} |\n v|^{m-2}|\n
(v-\wt{v})^+|^2 dx,
$$
yields
\begin{equation*}
J(\O_0)\leq C_{\O_0} J(\O_0).
\end{equation*}
Arguing as before for the case $p,m>2$, by choosing $\d$ sufficiently small
that $C_{\O_0}<1$, we get the desired contradiction, concluding the proof.
\end{proof}

\medskip

Let us now recall the fundamental ingredients of the moving plane method.
Let $\O$ be a bounded smooth domain contained in $\R^N$. Let us
consider a direction, say $x_1$ for example.  We set
$$
T_\lambda:=\{x \in \R^N: x_1= \lambda \}.
$$
Given $x\in \R^N$ and $\lambda <0$ for semplicity, we define
$$
x_\lambda:= (2\lambda-x_1, x_2,\ldots,x_N),\quad
u_\lambda(x):=u(x_\lambda),\quad
v_\lambda(x):=v(x_\lambda),
$$
and
$$
\O_\lambda:=\big \{x\in \O: x_1<\lambda \big\}.
$$
We also set
\begin{equation}
	\label{eq:capitalLambda}
	\Lambda :=\sup \big\{\lambda \in \R: \text{$x\in \O_t$ implies
	$x_\lambda \in \O$ for all $t\leq \lambda$} \big \},
	\qquad a:=\inf_{x\in \O} x_1.
\end{equation}
Finally, we set
\begin{equation*}
Z_{u,\lambda}:=\{x\in \Omega _{\lambda}: \nabla u(x)=\nabla u_{\lambda}(x)=0\},  \quad
Z_{v,\lambda}:=\{x\in \Omega _{\lambda}:\nabla v(x)=\nabla v_{\lambda}(x)=0\}.
\end{equation*}

We have the following

\begin{proposition}
Assume that~\eqref{libandpos} and~\eqref{coop} hold, and $1<p,m<\infty$.\\
Let $(u,v)\in C^{1,\a}(\ov{\O})\times C^{1,\a}(\ov{\O})$ be a solution to system~\eqref{eq:sist} and let
$\Lambda$ be as in~\eqref{eq:capitalLambda}. Then, for any $a\leq\lambda\leq \Lambda$, we have
\begin{equation}
	\label{zxclu:1.8222}
u(x)\leq u_{\lambda}(x)\quad\text{and}\quad v(x)\leq v_{\lambda}(x),\qquad \text{for all $x \in
\Omega _{\lambda}$}.
\end{equation}
Moreover, for any $\lambda$ such that $a<\lambda <\Lambda$, we have
\begin{equation}\label{zxcNEVFURu}
u(x) < u_{\lambda}(x),\qquad \text{for all $x \in \Omega_{\lambda}\setminus Z_{u,\lambda}$},
\end{equation}
and
\begin{equation}\label{zxcNEVFURv}
v(x) < v_{\lambda}(x),\qquad \text{for all  $x \in \Omega _{\lambda}\setminus Z_{v,\lambda}$}.
\end{equation}
Finally, it holds
\begin{equation}\label{zxclu:1.9222u}
\frac{\partial u}{\partial x_1}(x) \geq 0, \qquad \text{for all $x \in \Omega _{\Lambda}$},
\end{equation}
where $Z_u = \{x \in \Omega : \nabla u (x)=0\}$, and
\begin{equation}
	\label{zxclu:1.9222v}
\frac{\partial v}{\partial x_1}(x) \geq 0, \qquad \text{for all $x \in \Omega _{\Lambda}$}.
\end{equation}
\end{proposition}
\begin{proof}
For $a<\lambda<\Lambda$ and $\lambda$ sufficiently close to $a$, we  assume that
$\mathcal{L}(\O_\lambda)$ is as small as we need. In particular, we may assume that
Proposition~\ref{pro:confr} works with $\Omega_0=\Omega_\lambda$. Therefore, we set
$$
W_\lambda :=u-u_\lambda
\quad
\text{and}
\quad
H_\lambda :=v-v_\lambda,
$$
and we observe that, by construction, we have
\begin{equation*}
W_\lambda \leq 0\quad\text{on $\partial\O_\lambda$}
\quad
\text{and}
\quad
H_\lambda \leq 0\quad\text{on $\partial\O_\lambda$}.
\end{equation*}
In turn, by Proposition~\ref{pro:confr}, it follows that
\begin{equation*}
W_\lambda \leq 0\quad\text{in $\O_\lambda$}
\quad
\text{and}
\quad
H_\lambda \leq 0\quad\text{in $\O_\lambda$}.
\end{equation*}
We now define the set
\begin{equation}
\Lambda_0^{u,v}=\big\{\lambda > a: \text{$u\leq u_t$ and $v\leq v_t$ for all $t\in(a,\lambda]$}\big\}.
 \end{equation}
and
 \begin{equation}
\lambda_0=\sup\Lambda_0^{u,v}.
\end{equation}
Note that by continuity, we have $u\leq u_{\lambda_0}$ and $v\leq v_{\lambda_0}$.
We have to show that actually $\lambda_0 =\Lambda$. Hence, assume that by contradiction
$\lambda_0 <\Lambda$ and argue as follows. Let  $A$ be an open set such that
$$
Z_u\cap\Omega_{\lambda_0} \subset A \subset\Omega_{\lambda_0}
$$
and
$$
Z_v\cap\Omega_{\lambda_0} \subset A \subset\Omega_{\lambda_0}.
$$
Note that since $|Z_u|=|Z_v|=0$ (see~\cite[Theorem 2.2]{DS3} and the references therein), we can choose $A$ as small as we like.
Notice now that, since $f$ and $g$ are locally Lipschitz continuous, there exists a
positive constant $\Lambda$ such that
\begin{equation}\label{pranzo}
\frac{\partial f}{\partial s}(s,t)+\Lambda \geq 0,
\quad
\text{and}
\quad
\frac{\partial g}{\partial t}(s,t)+\Lambda \geq 0,
\qquad
\text{for all $s,t>0$}.
\end{equation}
Furthermore, $\frac{\partial f}{\partial t}(s,t)$ and $\frac{\partial g}{\partial s}(s,t)$
are non-negative for $s,t>0$, by assumption. Consequently,
\begin{align}
	\label{eq:confr}
-\Delta_p u + \Lambda u &=f(u,v)+ \Lambda u \leq  f(u_\lambda,v_\lambda)+\Lambda u_\lambda
= -\Delta_p u_\lambda + \Lambda u_\lambda , \\
\label{eq:confr2}
-\Delta_m v + \Lambda v &=g(u,v) + \Lambda v
\leq g(u_\lambda,v_\lambda)  + \Lambda v_\lambda
= -\Delta_m v_\lambda + \Lambda v_\lambda,
\end{align}
for any  $a\leq \lambda\leq \lambda_0$. In light
of~\eqref{eq:confr}-\eqref{eq:confr2}, we are able to write
\begin{equation}\label{eq:sist1}
\begin{cases}
-\D_p u  + \Lambda u\leq -\D_p u_\lambda  + \Lambda u_\lambda & \text{in $\O_\lambda$},\\
\qquad   u \leq u_\lambda & \text{in $\O_\lambda$}, \\
-\D_m v\, +\, \Lambda v\leq -\D_m v_\lambda \, +\, \Lambda v_\lambda & \text{in $\O_\lambda$}, \\
\qquad  v \leq v_\lambda & \text{in $\O_\lambda$}.
\end{cases}
\end{equation}
Then, by~\eqref{eq:sist1}, and a strong comparison principle \cite[Theorem 1.4]{Da2}, we get
$$
u<u_{\lambda_0} \qquad\text{or}\qquad  u\equiv u_{\lambda_0},
$$
in any connected
 component of $\Omega_{\lambda_{0}}\setminus Z_u$, and
$$
v < v_{\lambda_0}\qquad\text{or}\qquad v\equiv v_{\lambda_0},
$$
in any connected component of $\Omega_{\lambda_{0}}\setminus Z_u$.
We claim that
\begin{center}
The case $u\equiv u_{\lambda_0}$ in some  connected component $\mathcal{C}$ of $\Omega_{\lambda_{0}}\setminus Z_u$ is not possible.
\end{center}
In fact, by construction, it is $\partial \mathcal{C}\setminus T_{\lambda_0}\subseteq Z_u$.
If $u\equiv u_{\lambda_0}$, also the reflection of $\partial \mathcal{C}\setminus T_{\lambda_0}$
with respect to $T_{\lambda_0}$ in contained in $Z_u$. Consequently $\Omega\setminus Z_u$ would not be connected
that is a contradiction (see \cite{DS1,DS2}).
Consequently
\begin{equation}
	\label{pippopappo1}
u<u_{\lambda_0},
\end{equation}
in any connected component of $\Omega_{\lambda_{0}}\setminus Z_u$.
In the very same way, we get
\begin{equation}
	\label{pippopappo}
v<v_{\lambda_0}
\end{equation}
in any connected component of $\Omega_{\lambda_{0}}\setminus Z_v$.
Consider now a compact set $K$ in $\Omega_{\lambda_{0}}$ such that
${\mathcal L}(\Omega_{\lambda_{0}}\setminus K)$ is sufficiently small so that
Proposition~\ref{pro:confr} can be applied. By what we proved before, $u_{\lambda_{0}}-u$ and
$v_{\lambda_{0}}-v$ are positive in $K \setminus A$, which is compact. Then, by continuity,
we find $\epsilon>0$ such that, $\lambda_0+\epsilon < \Lambda$ and for
$\lambda<\lambda_0+\epsilon$ we have that ${\mathcal L}(\Omega_{\lambda}\setminus (K\setminus A))$ is still sufficiently
small as before, and $u_{\lambda}-u >0$ in $K \setminus A$, $v_{\lambda}-v > 0$ in $K \setminus
A$. In particular $u_{\lambda}-u >0$ and $v_{\lambda}-v >0$ on $\partial(K \setminus A)$.
Consequently $u \leq u_\lambda$ and $v \leq v_\lambda$ on
$\partial(\Omega_\lambda\setminus(K\setminus A))$. By Proposition~\ref{pro:confr} it follows
$u\leq u_\lambda$ and $v \leq v_\lambda$ in $\Omega_\lambda\setminus(K\setminus A)$ and,
consequently in $\Omega_\lambda$, which contradicts the assumption $\lambda_0<\Lambda$.
Therefore $\lambda_0 \equiv \Lambda$ and the thesis is proved.
The proof of~\eqref{zxcNEVFURu} and~\eqref{zxcNEVFURv} follows by the strong
comparison theorem exploited as above immediately as above, see~\eqref{pippopappo1} and~\eqref{pippopappo}.
Finally~\eqref{zxclu:1.9222u} and~\eqref{zxclu:1.9222v} follow by the monotonicity of the
solution, which is implicit in the above arguments.
\end{proof}

\medskip
\subsection{Proof of Theorem~\ref{main2}}

First, we give the following definition (cf.~\cite{DS1,DS2,DS3}).

\begin{definition}
Let $\rho \in L^{1}(\Omega)$ and $1\leq q<\infty$. The space $H^{1,q}_\rho(\Omega)$ is
defined as the completion of $C^1(\overline{\Omega})$ (or $C^{\infty }(\overline{\Omega})$) under the norm
  \begin{equation}\label{hthInorI}
\| v\|_{H^{1,q}_\rho}= \| v\|_{L^q (\Omega)}+\| \nabla v\|_{L^q (\Omega, \rho)},
\end{equation}
 where
$$
\|\nabla v\|^q_{L^p (\Omega, \rho)}:=\int_{\Omega}|\nabla v(x)|^q \rho(x) dx.
$$
We also recall that $H^{1,q}_{\rho}(\Omega)$ may be equivalently defined as the space of
functions having  distributional derivatives represented by a function for which the norm defined in
\eqref{hthInorI} is bounded. These two definitions are equivalent if the domain has piecewise regular boundary.
\end{definition}

\vskip2pt
\noindent
If $(u,v)\in C^1(\overline{\Omega})\times C^1(\overline{\Omega} )$ is a weak
solution of~\eqref{eq:sist}, then we have
\[
L_{(u,v)}((u_{x_i},v_{x_j}), (\varphi ,\psi)) \equiv (L^1_{(u,v)}((u_{x_i},v_{x_j}),(\varphi ,\psi )),
L^2_{(u,v)}((u_{x_i},v_{x_j}),(\varphi ,\psi)),
\]
where we have set, for $1<p,m <\infty$,
\begin{align*}
L^1_{(u,v)}((u_{x_i},v_{x_j}),(\varphi ,\psi ))& =
\int_{\Omega}|\nabla u|^{p-2}(\nabla u_{x_i},\nabla \varphi)
 +(p-2)\int_{\Omega}|\nabla u|^{p-4}(\nabla u,\nabla u_{x_i})(\nabla u,\nabla \varphi) \\
&-\int_{\Omega}\Big[\frac{\partial f}{\partial s}(u,v) u_{x_i}+ \frac{\partial f}{\partial t}(u,v) v_{x_i}\Big]\varphi \,dx,
\end{align*}
and
\begin{align*}
L^2_{(u,v)}((u_{x_i},v_{x_j}),(\varphi ,\psi )) &=
\int_{\Omega}|\nabla v|^{m-2}(\nabla v_{x_i},\nabla \psi)+(m-2)\int_{\Omega}|\nabla v|^{m-4}(\nabla v,\nabla v_{x_i})(\nabla v,\nabla \psi)  \\
&- \int_{\Omega}\Big[\frac{\partial g}{\partial s}(u,v) u_{x_i}+ \frac{\partial g}{\partial t}(u,v) v_{x_i}\Big]\psi \,dx,
\end{align*}
for any $\varphi,\psi\,\in C^1_0 (\Omega)$.
Moreover, the following equation holds
\begin{equation}
	\label{zxcLI:VISY}
L_{(u,v)}((u_{x_i},v_{x_j}), (\varphi ,\psi)) =0,
\qquad\text{for all $(\varphi ,\psi)$ in $H^{1,2}_{0, \rho _u}(\Omega)\times
H^{1,2}_{ 0,\rho _v}(\Omega)$},
\end{equation}
and all $i,j=1,\ldots ,N$, where
$$
\rho_u(x):= |\n u(x)|^{p-2},\qquad \rho_v(x):=|\n v(x)|^{m-2}.
$$
More generally, if $(w,h)\in H^{1,2}_{ \rho _u}(\Omega)\times
H^{1,2}_{\rho _v}(\Omega)$, we can define
$L_{(u,v)}((w,h), (\varphi ,\psi))$ as above.

\vskip6pt
\noindent
An immediate consequence is the following

\begin{theorem}
	\label{zxcPMFGSY}
Assume that~\eqref{libandpos} and~\eqref{coop} hold and that $\frac{2N+2}{N+2}<p,m<\infty$.
Let
$$
(w,h)\in H^{1,2}_{ \rho _u}\cap C(\Omega) \times H^{1,2}_{\rho _v}\cap C(\Omega)
$$
be a nonnegative weak solutions of
$$
L_{(u,v)}((w,h), (\varphi ,\psi))=0,\qquad\forall\varphi,\psi\,\in C^1_0 (\Omega).
$$
Then, for any domain $\Omega '\subset\Omega$ with $w\geq 0$ in $\Omega '$ and $h\geq 0$ in $\Omega '$,
one of the following four cases occurs
\begin{itemize}
 \item [(i)] $w>0$ and $h \equiv 0$ in $\O'$;
  \item [(ii)] $w>0$ and $h>0$ in $\O'$;
  \item [(iii)] $w \equiv 0$ and $h>0$ in $\O'$;
 \item [(iv)] $w\equiv 0$ and $h\equiv 0$ in $\O'$.
 \end{itemize}
\end{theorem}
\begin{proof}
In light of \eqref{coop}, we have  $\frac{\partial f}{\partial t}(s,t)$ and $\frac{\partial g}{\partial s}(s,t)$
are non-negative for $s,t>0$. Then, taking into account \eqref{pranzo}, it follows that
 $w$ and $h$
are nonnegative functions solving the  inequalities
\begin{align*}
& \int_{\Omega} |\nabla u|^{p-2}(\nabla w,\nabla \varphi)
+(p-2)\int_{\Omega}|\nabla u|^{p-4}(\nabla u,\nabla w)(\nabla u,\nabla \varphi) \,dx+\int_{\Omega}\Lambda w\varphi\,dx\geq 0,   \\
& \int_{\Omega} |\nabla v|^{m-2}(\nabla h,\nabla \psi)
+(m-2)\int_{\Omega}|\nabla v|^{m-4}(\nabla v,\nabla h)(\nabla v,\nabla\psi)\, dx+\int_{\Omega}\Lambda v\psi\,dx\geq 0,
\end{align*}
for all nonnegative test functions $\varphi$ and $\psi$, where  $\Lambda$ is the constant appearing in \eqref{pranzo}.
Therefore,  we can apply~\cite[Theorem 1.1]{DS2} to
$w$ and to $h$ separately obtaining that,
for every $s>1$ sufficiently close to $1$,
there exist positive constants $C_1 ,C_2$ such that
\begin{equation}
	\label{CLAIMIIwINTY}
\|w\|_{L^s (B(x, 2\delta))}\leq C_1\inf_{B(x, \delta)}w
\quad
\text{and}
\quad
\|h\|_{L^s (B(x, 2\delta))}\leq C_2\inf_{B(x, \delta)}h.
\end{equation}
Then, in turn, the sets $\{x\in\Omega': w(x)=0\}$ and $\{x\in\Omega': h(x)=0\}$ are
both closed (by continuity) and open (via inequalitites~\eqref{CLAIMIIwINTY})
in the domain $\Omega'$, yielding the assertion.
\end{proof}

We have the following

\begin{proposition}
Let $(u,v)\in C^{1,\a}(\ov{\O})\times  C^{1,\a}(\ov{\O})$ be a solution to
system~\eqref{eq:sist} and let $\Lambda$ be as in
\eqref{eq:capitalLambda}. Assume that~\eqref{libandpos} and~\eqref{coop} hold and that $\frac{2N+2}{N+2}<p,m<\infty$.
Then, for any $a\leq\lambda\leq \Lambda$, we have
\begin{equation}\label{zxclu:1.8222bis}
u(x)< u_{\lambda}(x)
\quad\text{and}\quad
v(x)< v_{\lambda}(x),\qquad \text{for all $x \in \Omega _{\Lambda}$}.
\end{equation}
Also
\begin{equation}\label{zxclu:1.9222ubis}
\frac{\partial u}{\partial x_1}(x) > 0, \qquad \text{for all $x \in \Omega _{\Lambda}$},
\end{equation}
and
\begin{equation}\label{zxclu:1.9222vbis}
\frac{\partial v}{\partial x_1}(x) > 0, \qquad \text{for all $x \in \Omega _{\Lambda}$}.
\end{equation}
\end{proposition}

\begin{proof}
	To prove~\eqref{zxclu:1.8222bis} it is sufficient to apply equations~\eqref{eq:confr} and~\eqref{eq:sist1}.
Instead to get~\eqref{zxclu:1.9222ubis} and~\eqref{zxclu:1.9222vbis} we use equations~\eqref{zxclu:1.9222u} and \eqref{zxclu:1.9222v}, together with Theorem~\ref{zxcPMFGSY}.
\end{proof}

\medskip
\subsection{Proof of Theorem \ref{main3}}
For  any given $x \in \R$, by Hopf boundary Lemma, (see~\cite{PSB}), it follows that
$$
u_y(x,0)=\frac{\partial u}{\partial y}(x,0) > 0
\quad
\text{and}
\quad
v_y(x,0)=\frac{\partial v}{\partial y}(x,0) > 0.
$$
We can therefore fix  $x_0$ and  $r$   such that
\begin{equation}\label{moster}
\frac{\partial u}{\partial y}(x,y)\geq \gamma > 0,\quad\frac{\partial v}{\partial y}(x,y)\geq \gamma > 0\quad \text{for all $(x,y) \in B_{2r}(x_0)$}\cap\{y\geq 0\},
\end{equation}
for some $\gamma >0$.
 Now, it follows that, for $\lambda\leq r$ fixed,
we have $\frac{\partial u}{\partial y}(x_0,y) > 0$ and $\frac{\partial v}{\partial y}(x_0,y) > 0$,
provided $0\leq y \leq \lambda$ and for every $0<\lambda'\leq \lambda$ we get
$u(x_0,y)<u(x_0,2\lambda'-y)$ and $v(x_0,y)<v(x_0,2\lambda'-y)$, provided that
$y\in [0,\lambda')$. Therefore we can exploit Theorem~\ref{intermedio} in the appendix and get that
for every $0<\lambda'\leq \lambda$ we have  $u(x_0,y)< u(x_0\,,\,2\lambda'-y)$ and $v(x_0,y)< v(x_0\,,\,2\lambda'-y)$ in $\Sigma_{\lambda'}\equiv \{(x,y): 0<y<\lambda'\}$.
Let us set
$$
\Lambda=\{\lambda\in \mathbb{R}^+:\text{$u<u_{\lambda'}$ and $v<v_{\lambda'}$ in $\Sigma_{\lambda'}$,
for all $\lambda'\leq\lambda$}\}.
$$
Define
$$
\bar{\lambda}=\sup\Lambda.
$$

We will prove the theorem proving that $\bar{\lambda}=\infty$. Note that, by continuity $u\leq u_{\bar{\lambda}}$
and $v\leq v_{\bar{\lambda}}$ in $\Sigma_{\bar{\lambda} }$  and also
$u< u_{\bar{\lambda}}$ and $v< v_{\bar{\lambda}}$, by the strong comparison principle.
Moreover by the above arguments we have
 $\frac{\partial u}{\partial y}(x,y)\geq 0$ and
$\frac{\partial u}{\partial y}(x,y)\geq 0$ in
$\Sigma_{\bar{\lambda}}$. Furthermore, by the strong maximum
principle for the linearized operator (see Theorem \ref{zxcPMFGSY}),
it follows that
$$
\frac{\partial u}{\partial y}(x,y)>0
\quad \text{and}\quad
\frac{\partial v}{\partial y}(x,y)>0,
$$
in $\Sigma_{\bar{\lambda}}$. To prove  that  $\bar{\lambda}=\infty$, let us argue by contradiction, and assume $\bar{\lambda}<\infty$. First of all let us show that
there exists some $\bar{x}\in\mathbb{R}$  such that
$$
\frac{\partial u}{\partial y} (\bar{x}\, ,\,\bar{\lambda})> 0
\quad\text{and}\quad
\frac{\partial v}{\partial y} (\bar{x}\, ,\,\bar{\lambda})> 0.
$$
Note that by continuity
$\frac{\partial u}{\partial y}(x,\bar{\lambda}),\frac{\partial v}{\partial y}(x,\bar{\lambda})\geq 0$.

Let us first show that there exists a point $x_0$ where $\frac{\partial u}{\partial y} (x_0\, ,\,\bar{\lambda})>0$. To prove this
we argue by contradiction and assume that

$$
\frac{\partial u}{\partial y} (x\, ,\,\bar{\lambda})=0
$$
for every $x\in \R$. Now, consider the function $u^\star (x,y)$
defined in $\Sigma_{2\bar{\lambda}}$ by
\[
u_\star (x,y)\equiv
\begin{cases}
u(x,y)        &    \text{if $0\leq y\leq \bar{\lambda}$,}  \\
u(x,2\bar{\lambda}-y) & \text{if $\bar{\lambda}\leq y\leq 2\bar{\lambda}$},
\end{cases}
\]
and consider the function $u_{\star}(x,y)$ defined in $\Sigma_{2\bar{\lambda}}$ by
\[
u^\star (x,y)\equiv
\begin{cases}
u(x,2\bar{\lambda}-y)     &\text{if $0\leq y\leq \bar{\lambda}$},   \\
u(x,y)           & \text{if $\bar{\lambda}\leq y\leq 2\bar{\lambda}$}.
\end{cases}
\]
Note that $u_\star$ is the even reflection of $u|_{\Sigma_{\bar{\lambda}}}$ and $u^\star$ is the even reflection of $u|_{\Sigma_{2\bar{\lambda}}\setminus \Sigma_{\bar{\lambda}}}$.
Also let  $v^\star$ and $v_{\star}$ defined in a similar fashion.  \\
Since we are assuming that $\frac{\partial u}{\partial y} (x\, ,\,\bar{\lambda})=0$
for every $x\in \mathbb{R}$, it follows that $u^\star$ and $u_{\star}$ are $C^1$ solutions of $-\Delta_m u^\star=f(u^\star,v^\star)$ and $-\Delta_m u_\star=f(u_\star,v_\star)$ respectively.
Since by definition $u<u_{\bar{\lambda}}$  and $v<v_{\bar{\lambda}}$ in $\Sigma_{\bar{\lambda}}$, we have
\[
u_\star\leq u^\star\quad \text{and}\quad
v_\star\leq v^\star
\]
in $\Sigma_{2\bar{\lambda}}$. Also $u_\star$ does not coincide with $u^\star$
because of the strict inequality $u<u_{\bar{\lambda}}$ in $\Sigma_{\bar{\lambda}}$.
Also, arguing as in \eqref{eq:confr} (see also \eqref{eq:confr2}), we find $\Lambda >0$ sufficiently large such that
$$
-\Delta_p u_\star +\Lambda u_\star\leq -\Delta_p u^\star +\Lambda u^\star
$$
Since $u_{\star}(x,\bar{\lambda})=u^{\star}(x,\bar{\lambda})$ for any $x\in\mathbb{R}$, by the strong comparison principle (see \cite[Theorem 1.4]{DS2}) it would follow that $u_\star\equiv u^{\star}$ in $\Sigma_{2\bar{\lambda}}$. This contradiction actually proves that there exists some $x_0\in\mathbb{R}$  such that $\frac{\partial u}{\partial y} (x_0\, ,\,\bar{\lambda})> 0$.

Let now $x_0\in\mathbb{R}$  such that $\frac{\partial u}{\partial y} (x_0\, ,\,\bar{\lambda})> 0$, and consider an interval $[x_0-\delta\, ;\, x_0+\delta]$ where $u_y$ is still strictly positive.
We claim that there exists $\bar{x}\in [x_0-\delta\, ;\, x_0+\delta]$
such that $\frac{\partial v}{\partial y} (\bar{x}\, ,\,\bar{\lambda})> 0$.
To prove this, assume by contradiction that $\frac{\partial v}{\partial y} (x\, ,\,\bar{\lambda})= 0$
for every $x\in[x_0-\delta\, ;\, x_0+\delta]$
and consider  $v^\star$ and $v_{\star}$ as above. Exploiting the strong comparison principle exactly as above in $\{(x,y)\,|\, x\in[x_0-\delta\, ;\, x_0+\delta]\}$, we get a contradiction.
Therefore we conclude that there exists a $\bar{x}$ such that $\frac{\partial v}{\partial y} (\bar{x}\, ,\,\bar{\lambda})> 0$. For such
$\bar x$ we therefore have
$$
\frac{\partial u}{\partial y} (\bar{x}\, ,\,\bar{\lambda})> 0
\quad\text{and}\quad
\frac{\partial v}{\partial y} (\bar{x}\, ,\,\bar{\lambda})> 0.
$$
Since now we have proved that $\frac{\partial u}{\partial y}(x_0 \, , \,  y)>0$ and $\frac{\partial v}{\partial y}(x_0 \, , \,  y)>0$ for every $y \in [0, \bar{\lambda} ]$,
it follows that we can find $\varepsilon >0$ such that
\begin{itemize}
\item[a)]$\frac{\partial u}{\partial y}(x_0\,,\,y)>0$ and $\frac{\partial v}{\partial y}(x_0\,,\,y)>0$ for every $y\in [0,\bar{\lambda}+\varepsilon]$
\item[b)] For every $0<\lambda'\leq \bar{\lambda}+\varepsilon$ we get  $u(x_0,y)<u(x_0\,,\,2\lambda'-y)$ and $v(x_0,y)<v(x_0\,,\,2\lambda'-y)$ provided that $y\in [0\, ,\, \lambda')$.
\end{itemize}
Note that $a)$ follows easily by the continuity of the derivatives. The proof of $b)$ is standard in the moving plane technique.
By Theorem \ref{intermedio} we now get that $u<u_{\lambda'}$ and  $v<v_{\lambda'}$ for every $0<{\lambda'}<\bar{\lambda}+\varepsilon$ which implies
$\sup \Lambda>\bar{\lambda}$, a contradiction. Therefore $\bar{\lambda}=\infty$.

\medskip
\section*{appendix}
We state and prove here a theorem which follows some  ideas contained in~\cite{DS4}.
For the readers convenience we provide a blueprint of the proof, which is also based
on  Proposition~\ref{pro:confr}.

\begin{theorem}
	\label{intermedio}Assume that~\eqref{libandpos} and~\eqref{coop} hold, and
let $(u,v)$ be a  weak $C^{1,\alpha}_{{\rm loc}}(\DD)\times C^{1,\alpha}_{{\rm loc}}(\DD)$ solution of~\eqref{eq:sist:semi}.
Assume that $\frac{3}{2}<p,m<\infty$.
Let $x_0\in \mathbb{R}$ and $\lambda\in \mathbb{R}$ fixed, and assume that
\begin{itemize}
\item[a)]$\frac{\partial u}{\partial y}(x_0\,,\,y)>0$ and $\frac{\partial v}{\partial y}(x_0\,,\,y)>0$ for every $y\in [0,\lambda]$
\item[b)] For every $0<\lambda'\leq \lambda$ we have  $u(x_0,y)<u(x_0\,,\,2\lambda'-y)$ and $v(x_0,y)<v(x_0\,,\,2\lambda'-y)$
(that is $u<u_{\lambda'}$, $v<v_{\lambda'}$) provided that $y\in [0\, ,\, \lambda')$.
\end{itemize}
Then, for every $0<\lambda'\leq \lambda$ and $(x,y)\in \Sigma_{\lambda'}$, it follows
\begin{center}
$u(x\, ,\,y)<u(x\, ,\,2\lambda'-y)\quad$ and $\quad v(x\, ,\,y)<v(x\, ,\,2\lambda'-y)$.
\end{center}
\end{theorem}

\begin{proof}
Let $L_\theta$ be the vector $(\cos \theta , \sin \theta)$ and $V_\theta$  the vector orthogonal to $L_{\theta}$ such that $(V_\theta,e_2)\geq 0$. We define $L_{x_0,s,\theta}$ the line parallel to $L_\theta$ passing through $(x_0,s)$.
We define $\mathcal{T}_{x_0,s,\theta}$ as the triangle delimited by $L_{x_0,s,\theta}$, $\{y=0\}$ and $\{x=x_0\}$, and we set
$
u_{x_0,s,\theta}(x)=u(T_{x_0,s,\theta}(x))$ and $v_{x_0,s,\theta}(x)=v(T_{x_0,s,\theta}(x))$,
where $T_{x_0,s,\theta}(x)$ is the point symmetric to $x$, w.r.t. $L_{x_0,s,\theta}$.
It is well known that $u_{x_0,s,\theta}$ and $v_{x_0,s,\theta}$ still are solutions of our system.
Also for simplicity we set $u_{x_0,s,0}=u_s$  and $v_{x_0,s,0}=v_s$.
Let us now consider $x_0\in \mathbb{R}$ and $\lambda\in \mathbb{R}$ fixed
as in the statement. We have the following
\vskip8pt
\noindent
{\bf Claim 1.}
{\em
There exists $\delta >0$ such that
for any $-\delta\leq \theta\leq\delta$ and for any $0< \lambda'\leq\lambda+\delta$
we have
$
u(x_0,y)<u_{x_0,\lambda',\theta}(x_0,y)$ and
$v(x_0,y)<v_{x_0,\lambda',\theta}(x_0,y)$
for every $0\leq y<\lambda'$}.\\

We argue by contradiction. If the claim were false, we could find
a sequence of  $\delta _n$ converging to $0$  and
$-\delta _n \leq\theta_n\leq \delta _n$, $0<\lambda_n\leq \lambda+\delta _n$,
$0\leq y_n< \lambda_n$ such that
$$
u(x_0,y_n)\geq u_{x_0,\lambda_n ,\theta_n}(x_0,y_n)
\quad\text{or}\quad
v(x_0,y_n)\geq v_{x_0,\lambda_n ,\theta_n}(x_0,y_n).
$$
For a sequence $y_n$, eventually considering a subsequence, we may assume that
$u(x_0,y_n)\geq u_{x_0,\lambda_n ,\theta_n}(x_0,y_n)$ for any $n\in\N$ or
$v(x_0,y_n)\geq v_{x_0,\lambda_n ,\theta_n}(x_0,y_n)$ for any $n\in\N$. Let us assume that
$u(x_0,y_n)\geq u_{x_0,\lambda_n ,\theta_n}(x_0,y_n)$ for any $n\in\N$.
At the limit, eventually considering subsequences,
 we may  assume that $\lambda_n$ converges to $\tilde{\lambda}\leq \lambda$. In addition
$y_n$ converges to $\tilde{y}$ for some $\tilde{y}\leq \tilde{\lambda}$. Let us show that  $\tilde{y}=\tilde{\lambda}$.
If  $\tilde{\lambda}=0$ it also follows  $\tilde{y}=\tilde{\lambda}=0$ since $0\leq y_n<  \lambda_n$.
If instead $\tilde{\lambda}>0$, by continuity it follows that $u(x_0,\tilde{y})\geq u_{\tilde{\lambda}}(x_0,\tilde{y})$.
Consequently
$y_n$ converges to $\tilde{\lambda}=\tilde{y}$ since  we know that $u<u_{\lambda'}$ for all $\lambda'\leq \bar{\lambda}$ in $\,\Sigma_{\lambda '}$.
By the mean value theorem since $u(x_0,y_n)\geq u_{x_0,\lambda_n ,\theta_n}(x_0,y_n)$, it follows that
$\frac{\partial u}{\partial V_{\theta_n}}(\tilde{x}_n,\tilde{y}_n)\leq 0$
at some point $\xi_n\equiv (\tilde{x}_n,\tilde{y}_n)$ lying on  the line from $(x_0,y_n)$ to $T_{x_0,\lambda_n,\theta_n}(x_0,y_n)$. We recall that
the vector $V_{\theta_n}$ is orthogonal to the line $L_{x_0,\lambda_n,\theta_n}$ and $V_{\theta_n}$ converges to $e_2$ since $\theta_n$ goes to $0$. Passing to the limit it follows that
$\frac{\partial u}{\partial y}(x_0,\tilde{\lambda})\leq 0$
which is impossible by the assumptions, proving the claim. \\

\noindent Let $\delta$ be the value provided by Claim 1.
\vskip8pt
\noindent
{\bf Claim 2.}
{\em
There is
$\rho=\rho(\delta)$ such that, for any $0<s\leq \rho$, the following inequalities hold:
$u < u_{x_0,s,\delta}$ in $\mathcal{T}_{x_0,s,\delta}$
($u < u_{x_0,s,-\delta}$ in $\mathcal{T}_{x_0,s,-\delta}$) and
$v < v_{x_0,s,\delta}$ in $\mathcal{T}_{x_0,s,\delta}$
($v < v_{x_0,s,-\delta}$ in $\mathcal{T}_{x_0,s,-\delta}$)}.\\

\noindent We prove that we can find $\rho=\rho(\delta)$ such that, for every $0<s\leq \rho$, it follows
$u < u_{x_0,s,\delta}$ in $\mathcal{T}_{x_0,s,\delta}$ and $v < v_{x_0,s,\delta}$ in $\mathcal{T}_{x_0,s,\delta}$.
If we replace $\delta$ by $-\delta$ the proof is exactly the same. To prove this, we can set $\rho$ in such a way that
\begin{enumerate}
\item[{\rm (i)}] $\rho<\lambda$, where $\lambda$ is given in the statement.
\item[{\rm (ii)}] For every $0<s\leq \rho$ we have $u\leq u_{x_0,s,\delta}$ on $\partial(\mathcal{T}_{x_0,s,\delta})$ and $v\leq v_{x_0,s,\delta}$ on $\partial(\mathcal{T}_{x_0,s,\delta})$.
\item [{\rm (iii)}] For $\rho$ small enough and $0<s\leq \rho$,
$\mathcal{L}(\mathcal{T}_{x_0,s,\delta})$ is so
small  to exploit Proposition \ref{pro:confr}.
 \end{enumerate}
Therefore,  given any $0<s\leq \rho$,  if we consider  $w_{x_0,s,\delta}=u-u_{x_0,s,\delta}$ and  $h_{x_0,s,\delta}=v-v_{x_0,s,\delta}$, we have that
 $w_{x_0,s,\delta}\leq 0$ and $h_{x_0,s,\delta}\leq 0$ on $\partial\mathcal{T}_{x_0,s,\delta}$ and
therefore, by Proposition \ref{pro:confr}, we get
$w_{x_0,s,\delta}\leq 0$ and $h_{x_0,s,\delta}\leq 0$ in $\mathcal{T}_{x_0,s,\delta}$. Also, by the strong
comparison principle exploited as above (see~\eqref{eq:sist1} and~\eqref{eq:confr}),
it follows that the strict inequalities hold. This concludes the proof of Claim 2.
\vskip8pt
\noindent
Consider now the values $\rho$ and $\delta $  provided by the Claims. Consider $0<\lambda'\leq\lambda$ and let us fix $0<\bar{s}<\min\{\rho,\lambda'\}$ so that by Claim 2 we have
  $w_{x_0,\bar{s},\delta}<0$ and $h_{x_0,\bar{s},\delta}\,< 0$ in $\mathcal{T}_{x_0,\bar{s},\delta}$.
We now define the continuous function $g(t)=(s(t),\theta (t)):[0,1] \to\R^2$, by
$s(t)=(t\lambda'+(1-t)\bar{s}$ and $\theta(t)=(1-t)\delta$, so that $g(0)=(\bar{s},\delta)$, $g(1)=(\lambda',0)$ and
$\theta (t)\neq 0$ for every $t\in [0,1)$. Moreover Claim 1 yields  $w_{x_0,\bar{s},\delta}\leq 0$ and $h_{x_0,s,\delta}\,\leq 0$ on $\partial(\mathcal{T}_{x_0,s(t),\theta(t)})$  for every $t\in [0,1)$.
Also $w_{x_0,s(t),\theta(t)}$ and $h_{x_0,s,\delta}$ are not identically zero on $\partial(\mathcal{T}_{x_0,s(t),\theta(t)})$,  for every $t\in [0,1)$.
We now let
$$
\overline{T}=\{\text{$\tilde{t}\in [0,1]$ such that $w_{x_0,\bar{s},\delta}\, ;\,h_{x_0,s,\delta}\,< 0$ in $\mathcal{T}_{x_0,s(t),\theta(t)}$ for every $0\leq t\leq\tilde{t}$}\},
$$
and $\bar{t}=\sup \overline{T},$
where, possibly, $\bar{t}=0$. Exploiting the moving-rotating plane technique as in  \cite{DS4} it follows that
$\bar{t}=1$, concluding the proof.
\end{proof}

\bigskip
\bigskip


\bigskip


\begin{thebibliography}{9999999}



\bibitem[ACM]{ACM} C.~Azizieh, P.~Cl\'ement, E.~Mitidieri,
\newblock Existence and a priori estimates for positive solutions of $p$-Laplace systems.
\newblock {\em J. Differential Equations} 184(2), 422--442, 2002.

\bibitem[AE]{AE} N.~Alikakos, L.C.~Evans,
\newblock Continuity of the gradient for weak solutions of a degenerate parabolic equation.
\newblock {\em J. Math. Pures Appl.} 62(3), 253--268, 1983.

\bibitem[BCN]{BCN} H.~Berestycki, L.~Caffarelli, L.~Nirenberg,
\newblock Further qualitative properties for elliptic equations in unbounded domains
\newblock {\em Ann. Scuola Norm. Sup. Pisa Cl. Sci. $(4)$} 25(1-2), 69--94, 1997.

\bibitem[BN]{BN} H.~Berestycki, L.~Nirenberg,
\newblock On the method of moving planes and the sliding method.
\newblock {\em Bulletin Soc. Brasil. de Mat Nova Ser} 22(1), 1--37, 1991.

\bibitem[BS]{BS} J.~Busca, B.~Sirakov,
\newblock Symmetry results for semilinear elliptic systems in the whole space.
\newblock {\em J. Differential Equations} 163(1), 41--56, 2000.

\bibitem[CFMT]{CFMT} P.~Cl\'ement, J.~Fleckinger, E.~Mitidieri, F.~de~Th\'elin,
\newblock Existence of positive solutions for a nonvariational quasilinear elliptic system.
\newblock {\em J.~Differential Equations} 166, 455--477, 2000.

\bibitem[Dam]{Da2} L.~Damascelli,
\newblock Comparison theorems for some quasilinear degenerate
elliptic operators and applications to symmetry and monotonicity results.
\newblock  {\em Ann. Inst. H. Poincar\'e. Analyse non lin\'eaire} 15(4), 493--516, 1998.

\bibitem[DS1]{DS1} L.~Damascelli, B.~Sciunzi.
\newblock Regularity, monotonicity and symmetry of positive solutions of $m$-Laplace equations.
\newblock {\em J. Differential Equations} 206(2), 483--515, 2004.

\bibitem[DS2]{DS2} L.~Damascelli, B.~Sciunzi,
\newblock Harnack inequalities, maximum and comparison principles, and
  regularity of positive solutions of $m$-Laplace equations.
\newblock {\em Calc. Var. Partial Differential Equations} 25(2), 139--159, 2006.

\bibitem[DS3]{DS3} L.~Damascelli, B.~Sciunzi,
\newblock Qualitative properties of solutions of $m$-Laplace systems.
\newblock {\em Adv.~Nonlinear Stud.}  5(2), 197--221, 2005.

\bibitem[DS4]{DS4} L.~Damascelli, B.~Sciunzi.
\newblock Monotonicity of the solutions of some quasilinear elliptic equations in the half-plane, and applications.
\newblock {\em J. Differential Equations} to appear.

\bibitem[De]{De} D.J.~de Figuereido,
\newblock Monotonicity and symmetry of solutions of elliptic systems in general domains.
\newblock {\em NoDEA Nonlinear Differential Equations Appl.}  1(2), 119--123, 1994.


\bibitem[GNN]{GNN} B.~Gidas, W.~M.~Ni, and L.~Nirenberg,
\newblock Symmetry and related properties via the maximum principle.
\newblock  {\em Comm. Math. Phys.} 68(3), 209--243, 1979.


\bibitem[PS3]{PSB} P.~Pucci, J.~Serrin,
\newblock \emph{The maximum principle}.
\newblock Birkhauser, Boston (2007).


\bibitem[Ser]{S} J.~Serrin,
\newblock A symmetry problem in potential theory.
\newblock  \emph{Arch. Rational Mech. Anal} 43(4), 304--318, 1971.

\bibitem[Tr]{Tr} W.C.~Troy,
\newblock Symmetry properties in systems of semilinear elliptic equations.
\newblock {\em J.\ Differential Equations} 42(3), 400--413, 1981.

\end{thebibliography}
\end{document}